\documentclass[12pt]{article}
\usepackage{amssymb, amsmath, amsthm}
\usepackage{epsfig}
\usepackage{subfigure}
\usepackage{color}
\allowdisplaybreaks

\begin{document}

\newtheorem{theorem}{Theorem}[section]
\newtheorem{definition}[theorem]{Definition}
\newtheorem{lemma}[theorem]{Lemma}
\newtheorem{observation}[theorem]{Observation}
\newtheorem{example}[theorem]{Example}
\newtheorem{corollary}[theorem]{Corollary}
\newtheorem{remark}[theorem]{Remark}
\newenvironment{defn}{\begin{definition}}{\hfill$\blacksquare$\end{definition}}
\newenvironment{defn*}{\begin{definition}}{\end{definition}}

\title{\bf The distance between  two limit $q$-Bernstein operators}

\author{Sofiya Ostrovska and Mehmet Turan}
\date{}
\maketitle

\begin{center}
{\it Atilim University, Department of
Mathematics,  Incek  06836, Ankara, Turkey}\\
{\it e-mail: sofia.ostrovska@atilim.edu.tr, mehmet.turan@atilim.edu.tr}\\
%{\it Tel: +90 312 586 8211,  Fax: +90 312 586 8091}
\end{center}

\begin{abstract}

For $q\in(0,1),$ let $B_q$ denote the limit $q$-Bernstein operator. In this paper, the distance between $B_q$ and $B_r$ for distinct $q$ and $r$ in the operator norm on $C[0,1]$ is estimated, and it is  proved that $1\leqslant \|B_q-B_r\|\leqslant 2,$ where both of the equalities can be attained.
To elaborate more, the distance depends on whether or not $r$ and $q$ are rational powers of each other. For example, if $r^j\neq q^m$ for all $j,m\in \mathbb{N},$ then $\|B_q-B_r\|=2,$ and if $r=q^m, m\in \mathbb{N},$ then $\|B_q-B_r\|=2(m-1)/m.$


\end{abstract}

\section{Introduction and statement of  results}

The limit $q$-Bernstein operator $B_q$ can be viewed as an analogue of the Sz\'{a}sz-Mirakyan operator pertinent the Euler distribution, also known as the $q$-deformed Poisson distribution, see \cite[Ch. 3, Sec. 3.4]{charalam} and \cite{india}.
The latter is used in the $q$-boson theory, which delivers a $q$-deformation of the
quantum harmonic oscillator formalism
\cite{bieden}.
Going into details, the $q$-deformed Poisson
distribution defines the distribution of energy in a
$q$-analogue of the coherent state \cite{bieden, jing}.
The $q$-analogue of the boson operator calculus is recognized as an indispensable area within theoretical physics.
It brings out explicit expressions for the representations of the quantum group
$SU_q(2),$ which plays an important role in various problems such as exactly solvable lattice models of statistical
mechanics, integrable model field theories, conformal field theory, only to mention a few. For additional information
the reader is referred to \cite{qgroup}. Therefore, linear operators related to the $q$-deformed
Poisson distribution, in particular the limit $q$-Bernstein operator, are of significant interest for applications.

The operator $B_q$ also emerges as a limit for a sequence of $q$-Bernstein polynomials in the case $0<q<1.$
Over the past years, the limit $q$-Bernstein operator
has been studied widely from different perspectives. Its
approximation,  spectral, and functional-analytic properties,
probabilistic interpretation, the behavior of iterates, and the
impact on the analytic characteristics of functions have been
examined. See, for example, \cite{nazimhigher, parametric, wangmeyer}.
The review of obtained results along with extensive bibliography can be found in \cite{JAM}.

Let $q>0.$ For any $a\in{\mathbb C},$ as given in \cite[Ch. 10]{askey}, we denote:
$$(a;q)_0:=1, \quad
(a;q)_k:=\prod_{s=0}^{k-1}(1-aq^s), \quad
(a;q)_\infty:=\prod_{s=0}^{\infty}(1-aq^s).$$

\begin{definition}$\cite{jat}$
For each $q\in(0,1),$ and $f\in C[0,1],$ the limit $q$-Bernstein operator is defined by $f\mapsto B_qf$ where
\begin{align}
B_q (f;x) =
\begin{cases}
\displaystyle \sum_{k=0}^{\infty} f(1-q^k)p_k(q;x) & \textrm{if } \: x\in[0,1) \\
f(1) & \textrm{if } \:  x = 1
\end{cases} \label{bq}
\end{align}
in which
\begin{align}
p_k(q;x)=\frac{x^k (x;q)_\infty}{(q;q)_k}. \label{pk}
\end{align}
\end{definition}

As it can be readily seen that $B_q(f;x)$ is defined by the values of $f$ on the set
\begin{align*}
{\mathbb J}_q:=\{1-q^k: k\in{\mathbb N}_0\}.
\end{align*}

Further, Euler's identity \cite[page 490, Corollary 10.2.2]{askey}
\begin{align} \label{euler}
\frac{1}{(x;q)_\infty} = \sum_{k=0}^{\infty} \frac{x^k}{(q;q)_k}, \quad |x|<1 , \quad |q|<1,
\end{align}
implies that
\begin{align}
\sum_{k=0}^\infty p_k(q;x)=
\begin{cases}
1, & x\in[0, 1) \\
0, & x=1.
\end{cases} \label{sumpk}
\end{align}

Formulae \eqref{bq} and \eqref{sumpk} show that $B_q$ is a positive linear operator on $C[0,1]$ with $\|B_q\|=1.$
Recently, the continuity of the operator $B_q$ with respect to the parameter has been examined in \cite{manal} where the outcome below has been presented.
\begin{theorem}
For every $f\in C[0,1],$ one has: $$\lim_{q\to a}B_q(f;x)= B_a(f;x)$$
and the convergence is uniform on $[0,1].$
\end{theorem}

This demonstrates the continuity of $B_q$ in strong operator norm. The aim of the current paper is to investigate whether the continuity persists with respect to the topology produced by the uniform operator norm. It turns out that in this topology, $\{B_q\}_{q\in(0,1)}$ forms a discrete set of operators where each $B_q$ is an isolated point so that
$\|B_q-B_r\| \geqslant 1$ whenever $q\neq r.$

The reasoning of the  present paper is based essentially on the next theorem,  which in itself can be of interest. The idea of the proof is attributed to a statement made by a Mathoverflow user  under the nickname `fedja', see \cite{fedja}.

\begin{theorem} \label{thmfed} Let $q\in (0,1)$ and $m\geqslant 2$ be an integer. Then, for all $x\in[0,1]$ and all $k\in{\mathbb N}_0,$ the following inequality holds:
\begin{align}
p_{mk}(q;x) \leqslant p_k(q^m;x). \label{eqfed}
\end{align}
\end{theorem}

Obviously, by the triangle inequality, $\|B_q-B_r\| \leqslant 2.$ In some cases, the equality is attained, as claimed by the below-mentioned result.

\begin{theorem}\label{thmqr2}
 Let $q,r\in[0,1].$ If ${\mathbb J}_q \cap {\mathbb J}_r=\{0\},$ then $\|B_q - B_r \| = 2.$
\end{theorem}

Now, comes the case ${\mathbb J}_q\cap{\mathbb J}_r\neq\{0\}.$ This situation occurs when $r^j = q^m$ for some positive integers $j$ and $m$ and reveals:

\begin{theorem} \label{thmBqrmj}
Let $0<r<q<1$ and $r^j=q^m$ where $j$ and $m$ are mutually prime positive integers.
Then $\|B_q - B_r \| \geqslant \frac{2(m-1)}{m}.$
\end{theorem}

In the case when $j=1,$ i.e., ${\mathbb J}_{q^m}={\mathbb J}_r\subset {\mathbb J}_q,$ the exact value of $\|B_q-B_r\|$ has been obtained.

\begin{theorem} \label{thmBqr1}
Let $q\in (0,1)$ and $r=q^{m}$ for some integer $m\geqslant 2.$ Then $\|B_q - B_r \| = \frac{2(m-1)}{m}.$
\end{theorem}

\begin{corollary} \label{corBqr}
For any $q\neq r,$ one has $1 \leqslant \|B_q-B_r\| \leqslant 2.$
The inequalities  are sharp in the sense that both equalities are attained.
\end{corollary}

\section{Some auxiliary results}

In this section, some   results which will  later contribute in proving our theorems,  are presented. To begin with, we point out  the following:
\begin{observation}\label{obs}
For any positive integer $m,$ one has:
\begin{align} \label{limpmk}
\lim_{x\to 1^-} \sum_{k=0}^{\infty} p_{mk}(q;x)=\frac{1}{m}.
\end{align}
\end{observation}
\begin{proof}
Clearly,
\begin{align*}
\sum_{k=0}^\infty p_{mk}(q;x)=(x;q)_\infty \sum_{k=0}^{\infty} \frac{x^{mk}}{(q;q)_{mk}} \leqslant
(x;q)_\infty \sum_{k=0}^{\infty} \frac{x^{mk}}{(q;q)_{k}} = \frac{(x;q)_\infty}{(x^m;q)_\infty}
\end{align*}
where the identity \eqref{euler} is used. Therefore,
\begin{align*}
\sum_{k=0}^\infty p_{mk}(q;x)\leqslant \frac{(x;q)}{(x^m;q)_\infty}=\frac{1-x}{1-x^m}\, \frac{(qx;q)_\infty}{(qx^m;q)_\infty} \to \frac{1}{m}\quad \text{as} \quad x\to 1^-.
\end{align*}
On the other hand,
\begin{align*}
\sum_{k=0}^\infty p_{mk}(q;x)&=(x;q)_\infty \sum_{k=0}^{\infty} \frac{x^{mk}}{(q;q)_{mk}} \geqslant
(x;q)_\infty \sum_{k=0}^{\infty} \frac{x^{mk}}{(q;q)_{\infty}}\\
&= \frac{(x;q)_\infty}{(q;q)_\infty (1-x^m)} =
\frac{1-x}{1-x^m}\, \frac{(qx;q)_\infty}{(q;q)_\infty} \to \frac{1}{m}\quad \text{as} \quad x\to 1^-.
\end{align*}
\end{proof}

\begin{lemma}\label{lemint} If, for every $k\in{\mathbb N},$ inequality \eqref{eqfed} holds for $x\in[0, 1-q^{mk+m/2}],$ then it holds for all $x\in[0,1].$
\end{lemma}
\begin{proof} Clearly, $[0,1]=\{1\}\bigcup_{k=1}^\infty [0, 1-q^{mk+m/2}].$ For $x=1,$ the inequality \eqref{eqfed} is obvious. Now, by \eqref{pk}, inequality \eqref{eqfed} can be expressed as
$$\frac{x^{mk}(x;q)_\infty}{(q;q)_{mk}} \leqslant \frac{x^k(x;q)_\infty}{(q^m;q^m)_k}$$
or, equivalently,
\begin{align}
u_k(x):=x^{(m-1)k}\prod_{j=0}^{k-1}\prod_{\ell=1}^{m-1} \frac{1}{1-q^{\ell+mj}} \leqslant
\prod_{j=0}^{\infty}\prod_{\ell=1}^{m-1} \frac{1}{1-q^{\ell+mj}x}. \label{umk}
\end{align}
Clearly,
$$
\max_{j\in{\mathbb N}} u_j(x)=u_k(x)
$$
on $[1-\alpha_{k-1}, 1-\alpha_k]$ where $\alpha_k$ can be found from the equation $u_{k+1}(x)=u_k(x).$ Therefore, if for every $k\in{\mathbb N},$ \eqref{umk} holds on $[0, 1-\alpha_k],$ then it holds on $[1-\alpha_{k-1}, 1-\alpha_k],$ and, as a result, on $[0,1].$ That is why, \eqref{umk} is going to be proved on $[0,1-q^{mk+m/2}] \supseteq [0, 1-\alpha_k]$ for every $k\in{\mathbb N}.$ To justify this inclusion, one has to show that $\alpha_k \geqslant q^{mk+m/2}.$ Indeed,
\begin{align*}
(1-\alpha_k)^{m-1}&=\prod_{\ell=1}^{m-1}(1-q^{mk+\ell}) \\
&= \sqrt{\prod_{\ell=1}^{m-1}\left[(1-q^{mk+\ell})(1-q^{mk+m-\ell})\right]} \\
&= \prod_{\ell=1}^{m-1} \sqrt{1-q^{mk+\ell}-q^{mk+m-\ell}+q^{2mk+m}} \\
& \leqslant \prod_{\ell=1}^{m-1} \sqrt{(1-q^{mk+m/2})^2} = (1-q^{mk+m/2})^{m-1}
\end{align*}
by virtue of the arithmetic-geometric mean inequality. This completes the proof.
\end{proof}
\begin{lemma}\label{lemrho} Let $\rho(t)=1/(e^t+e^{-t}-2),$ $t>0.$ Then, for all $s, t >0,$ the following inequality is valid:
\begin{align}
\rho(s+t) \leqslant e^{-s}\rho(t). \label{rho}
\end{align}
\end{lemma}
\begin{proof} Equivalently, one may prove that $1/\rho(s+t) \geqslant e^{s}/\rho(t),$ that is, $e^{-t}-2 \leqslant e^{-2s-t}-2e^{-s}.$ If $s=0,$ then both sides are equal, while for $s>0,$ the derivative of the right hand side with respect to $s$ is positive, which yields the statement.
\end{proof}

For the sequel, a special quadrature formula to approximate $\int_a^b f(t)dt$  is needed. More precisely, we set, for $m\geqslant 2,$
\begin{align} \label{qm}
Q_m(f;a,b):=\frac{b-a}{m-1}\sum_{j=1}^{m-1} f\left(a+\frac{b-a}{m}\,j\right).
\end{align}
It is not difficult to see that the quadrature formula gives the exact value of the integral for polynomials of degree at most 1. Denote by $R_{a,b}(f)$ the error in this approximation, i.e.,
\begin{align} \label{err}
R_{a,b}(f)=\int_a^b f(t)dt - Q_m(f;a,b).
\end{align}

\begin{lemma}
The error \eqref{err} is given by
\begin{align}\label{peano}
R_{a,b}(f)=\int_a^b K_{a,b}(t) f''(t)dt
\end{align}
where
\begin{align} \label{kab}
K_{a,b}(t)&=
\begin{cases}
\frac{1}{2}(t-a)^2 & {\rm if } \:\: t\in[a, a+h_1] \\
\frac{1}{2}\left(t-a-\frac{mh_1k}{m-1}\right)^2+\frac{mh_1^2k(m-k-1)}{2(m-1)^2} & {\rm if }\:\: t\in[a+kh_1, a+(k+1)h_1], 1\leqslant k \leqslant m-2\\
\frac{1}{2}(b-t)^2 & {\rm if } \:\: t\in[b-h_1,b]
\end{cases}
\end{align}
and $h_1=\frac{b-a}{m}.$
\end{lemma}
\begin{proof}
By Peano's Theorem (see, for example \cite[Theorem 3.2.3, page 123]{stoer}), the error is expressed by \eqref{peano}
where $K_{a,b}(t)=R_{a,b}((x-t)_+)$
and
\begin{align*}
(x-t)_+=
\begin{cases}
x-t & {\rm if } \:\: x \geqslant t, \\
0 & {\rm if } \:\: x < t.
\end{cases}
\end{align*}
Here, $(x-t)_+$ is considered as a function of $x.$ Plain calculation of $R_{a,b}((x-t)_+)$ using \eqref{err} results in \eqref{kab}.
\end{proof}

In what follows, given $h>0,$ denote by $K(t)$ the $h$-periodic function on ${\mathbb R}$ such that $K(t)=K_{0,h}(t)$ for $t\in[0,h]$
where $K_{a,b}(t)$ is given by \eqref{kab}. In other words, $K(t+h)=K(t)$ for all $t\in{\mathbb R}$ and
\begin{align} \label{kern}
K(t)=
\frac{1}{2}\, t^2-\frac{hk}{m-1} \, t+\frac{h^2k(k+1)}{2m(m-1)} \quad {\rm for }\:\: t\in[kh_1, (k+1)h_1],\:\: 0\leqslant k \leqslant m-1
\end{align}
where $h_1=h/m.$

\begin{lemma}
For all $m\geqslant 2,$ the following inequality holds:
\begin{align}
\int_0^h K(t) \frac{dt}{t^2} \geqslant 8 \int_h^{3h} K(t) \frac{dt}{t^2}. \label{intk}
\end{align}
\end{lemma}

\begin{proof}

For $j=0,1,2,$ set $$I_j:=\int_0^h \frac{K(t)}{(t+jh)^2}dt.$$
The inequality \eqref{intk} is equivalent to $I_0\geqslant 8(I_1+I_2).$
Now,
\begin{align*}
I_0 &=\int_{0}^{h} \frac{K(t)}{t^2} \, dt = \int_{0}^{h_1} \frac{K(t)}{t^2}\,dt+\sum_{k=1}^{m-1} \int_{kh_1}^{(k+1)h_1} \frac{K(t)}{t^2}\,dt \\
&=\frac{h}{2m}+\sum_{k=1}^{m-1}\left[\frac{h}{2m}-\frac{hk}{m-1}\ln\left(1+\frac{1}{k}\right)+\frac{h^2k(k+1)}{2m(m-1)}\left(\frac{1}{kh_1}-\frac{1}{(k+1)h_1}\right)\right]\\
&=h-\frac{h}{m-1}\sum_{k=1}^{m-1} k\ln\left(1+\frac{1}{k}\right)
\end{align*}
which gives
\begin{align*}
\frac{(m-1)I_0}{h} = m-1-m\ln m + \ln(m!).
\end{align*}
On the other hand, for $j\geqslant 1,$ using \eqref{kern} and the substitution $x=t+jh,$ one obtains
\begin{align*}
I_j &=\int_{0}^{h} \frac{K(t)}{(t+jh)^2} \, dt = \sum_{k=0}^{m-1} \int_{kh_1}^{(k+1)h_1} \frac{K(t)}{(t+jh)^2} \, dt.\\
&=\sum_{k=0}^{m-1} \int_{jh+kh_1}^{jh+(k+1)h_1} \left[\frac{1}{2}-h\left(j+\frac{k}{m-1}\right)\, \frac{1}{x}+h^2\left(\frac{j^2}{2}+\frac{jk}{m-1}+ \frac{k(k+1)}{2m(m-1)}\right)\frac{1}{x^2}\right]\,dx\\
&=\frac{h}{2}-h\sum_{k=0}^{m-1}\left(j+\frac{k}{m-1}\right)\ln\left(1+\frac{1}{jm+k}\right)+S_j
\end{align*}
where
\begin{align*}
S_j &= \frac{h^2}{h_1}\sum_{k=0}^{m-1}\left(\frac{j^2}{2}+\frac{jk}{m-1}+ \frac{k(k+1)}{2m(m-1)}\right)\left(\frac{1}{jm+k}-\frac{1}{jm+k+1}\right)\\
&=\frac{h}{2}\left[j^2\left(\frac{1}{j}-\frac{1}{j+1}\right)+\frac{1}{m-1}\sum_{k=0}^{m-1}[2mjk+k(k+1)]\left(\frac{1}{jm+k}-\frac{1}{jm+k+1}\right)\right]\\
&=\frac{h}{2}\left\{\frac{j}{j+1}+\frac{1}{m-1}\sum_{k=0}^{m-1}\left[1-jm(jm+1)\left(\frac{1}{jm+k}-\frac{1}{jm+k+1}\right)\right]\right\}\\
&=\frac{h}{2}\left[\frac{j}{j+1}+\frac{m}{m-1}-\frac{jm(jm+1)}{m-1}\left(\frac{1}{jm}-\frac{1}{jm+m}\right)\right]
=\frac{h}{2}.
\end{align*}
Therefore,
\begin{align*}
I_j =h-h\sum_{k=0}^{m-1}\left(j+\frac{k}{m-1}\right)\ln\left(1+\frac{1}{jm+k}\right)
\end{align*}
or
\begin{align*}
\frac{I_j}{h} =1-j\ln\left(1+\frac{1}{j}\right)-\frac{m}{m-1}\ln(jm+m)+\frac{\ln[(jm+m)!]-\ln[(jm)!]}{m-1}.
\end{align*}
With the help of Stirling's formula
\begin{align*}
\sqrt{2\pi} n^{n+1/2} e^{-n+1/(12n+1)} < n! < \sqrt{2\pi} n^{n+1/2} e^{-n+1/(12n)},
\end{align*}
one gets
\begin{align*}
\frac{I_1+I_2}{h} &= 2+\ln 2 -2\ln 3 - \frac{m[\ln(2m)+\ln(3m)]}{m-1}+
\frac{\ln[(3m)!]-\ln(m!)]}{m-1} \\
& \leqslant 2+\ln 2 -2\ln 3 - \frac{m[\ln(2m)+\ln(3m)]}{m-1}\\
& \qquad + \frac{1}{m-1}\left[\left(3m+\frac{1}{2}\right)\ln(3m)-\left(m+\frac{1}{2}\right)\ln(m)-2m+\frac{1}{36m}-\frac{1}{12m+1}\right]
\end{align*}
and as a result
\begin{align*}
\frac{(m-1)(I_1+I_2)}{h} \leqslant -2+\frac{5}{2} \ln 3 - \ln 2 + \frac{1}{36m}-\frac{1}{12m+1},
\end{align*}
while
\begin{align*}
\frac{(m-1)I_0}{h} \geqslant -1 + \ln{\sqrt{2\pi m}} + \frac{1}{12m+1}.
\end{align*}
The needed inequality $I_0\geqslant 8(I_1+I_2)$ for all $m\geqslant 2$ follows from the fact that
\begin{align*}
-1+ \ln\sqrt{2\pi m} +\frac{1}{12m+1} \geqslant 8\left(-2 + \frac{5}{2} \ln 3 - \ln 2 + \frac{1}{36m} - \frac{1}{12m+1}\right),
\end{align*}
or equivalently,
\begin{align*}
\ln\sqrt{2\pi m} +\frac{9}{12m+1} - \frac{2}{9m} + 15 - 20 \ln 3 + 8\ln 2  \geqslant 0.
\end{align*}
To see this, let
\begin{align*}
\theta(x)= \ln\sqrt{2\pi x} +\frac{9}{12x+1} - \frac{2}{9x} + 15 - 20 \ln 3 + 8\ln 2.
\end{align*}
Then,
\begin{align*}
\theta'(x)= \frac{1}{2x} - \frac{108}{(12x+1)^2} + \frac{2}{9x^2} > \frac{1}{2x} - \frac{108}{(12x)^2} + \frac{2}{9x^2}
= \frac{18x-19}{36x^2} > 0 \quad \text{for} \quad x\geqslant 2.
\end{align*}
Hence, for all $m\geqslant 2,$ one has:
$$\theta(m)\geqslant \theta(2) \approx 0.0073,$$
which completes the proof.
\end{proof}
\begin{remark} In the case $m=2,$ an alternative proof is presented in \cite{fedja}.
\end{remark}
\begin{corollary}
Let $\rho(t)$ be the function from Lemma \ref{lemrho}. Then
$$\int_0^h K(t) \rho(t) dt \geqslant 8 \int_h^{3h} K(t) \rho(t) dt.$$
\end{corollary}
\begin{proof}
The statement is a consequence of the fact that $t\mapsto \frac{1}{t^2 \rho(t)}=\frac{e^{t}+e^{-t}-2}{t^2}$ is an increasing function for $t>0.$
\end{proof}
\begin{corollary} \label{corh0}
For $h \leqslant h_0:=\ln 4,$
$$\int_0^h K(t) \rho(t) dt \geqslant e^{3h/2} \int_h^{3h} K(t) \rho(t) dt.$$
\end{corollary}
Next, for a given $f:[a,b]\to {\mathbb R},$ denote by $E_{a,b}$ the error in the composite quadrature formula to approximate $\int_a^b f(t) dt$ when the interval $[a,b]$ is divided into $n$ subintervals of equal length $h$ and the rule \eqref{qm} is applied on each subinterval. That is,
\begin{align}
E_{a,b}=\int_a^b f(t)dt- \sum_{j=1}^{n} Q_m(f;a+(j-1)h, a+jh), \label{errab}
\end{align}
where $h=(b-a)/n.$ If $b=\infty,$ we take $n=\infty.$

\begin{lemma} \label{lemeras}
Let $f(t)=-\ln(1-e^{-t}),$ $t>0.$ Then, for all $a>0$ and any step size, one obtains:
\begin{align*}
E_{s+a,\infty} \leqslant e^{-s} E_{a,\infty}.
\end{align*}
\end{lemma}
\begin{proof}
By Peano's Theorem on the integral representation of the error term, $$E_{s,\infty}=\int_s^{\infty} f''(t) K(t-s)dt,$$ where $K(t)$ is defined by \eqref{kern}. Since $f''(t)=\rho(t)$, from Lemma \ref{lemrho}, application of \eqref{rho} yields:
\begin{align*}
E_{s+a,\infty}=\int_{s+a}^{\infty} \rho(t) K(t-s-a) dt =\int_{a}^{\infty} \rho(t+s) K(t-a) dt \leqslant e^{-s} \int_a^\infty \rho(t)K(t-a) dt =e^{-s}E_{a,\infty}.
\end{align*}
\end{proof}
\begin{lemma} \label{intst}
If $f(t)=-\ln(1-e^{-t}),$ $t>0,$ then for all $S,T>0,$ there holds:
$$\int_S^\infty f(t) dt \geqslant -ST + \int_0^T f(t)dt.$$
\end{lemma}
\begin{proof}
It can be observed geometrically since $f(t)$ is a decreasing continuous function symmetric about the line $y=x.$
\end{proof}

\section{Proofs of  the main results}

\begin{proof}[Proof of Theorem \ref{thmfed}]
As it was done in the proof of Lemma \ref{lemint}, inequality \eqref{eqfed} is equivalent to
\begin{align*}
x^{(m-1)k}\prod_{j=0}^{k-1}\prod_{\ell=1}^{m-1} \frac{1}{1-q^{\ell+mj}} \leqslant
\prod_{j=0}^{\infty}\prod_{\ell=1}^{m-1} \frac{1}{1-q^{\ell+mj}x}.
\end{align*}
Taking the logarithm of both sides leads to
\begin{align*}
-(m-1)k\ln\left(\frac{1}{x}\right)+\sum_{j=0}^{k-1}\sum_{\ell=1}^{m-1} \ln\left(\frac{1}{1-q^{\ell+mj}}\right) \leqslant
\sum_{j=0}^{\infty}\sum_{\ell=1}^{m-1} \ln\left(\frac{1}{1-q^{\ell+mj}x}\right).
\end{align*}
Set $h=\ln(1/q^m),$ $S=\ln(1/x),$ $T=kh$ and $f(t)=-\ln(1-e^{-t}).$ Then the inequality becomes
\begin{align*}
-ST+\sum_{j=0}^{k-1} \frac{h}{m-1}\sum_{\ell=1}^{m-1}f(jh+\frac{h}{m}\ell)
\leqslant \sum_{j=0}^{\infty} \frac{h}{m-1}\sum_{\ell=1}^{m-1}f(S+jh+\frac{h}{m}\ell)
\end{align*}
which can be written as
\begin{align*}
-ST+\sum_{j=0}^{k-1} Q_m(f;jh,(j+1)h)
\leqslant \sum_{j=0}^{\infty} Q_m(f;S+jh,S+(j+1)h).
\end{align*}
The sums in the last inequality can be viewed as the composite quadrature formulas for the integrals $\int_0^T f(t)dt$ and $\int_S^\infty f(t)dt,$ respectively.
Therefore, by \eqref{errab}, one has
\begin{align*}
-ST+\int_0^T f(t)dt - E_{0,T}
\leqslant \int_S^\infty f(t)dt - E_{S,\infty}.
\end{align*}
Using Lemma \ref{intst}, if one can show that
\begin{align}
E_{0,T} \geqslant E_{S,\infty} \label{ineq}
\end{align}
for $h\leqslant h_0,$ the proof will be complete for $q\geqslant 1/2.$ Due to Lemma \ref{lemint}, we need only to deal with the case $x\in[0,1-q^{mk+m/2}]$ or
\begin{align}
e^{-S} \leqslant 1-e^{-T-h/2}. \label{est}
\end{align}
By Lemma \ref{lemeras}, one derives that $E_{S,\infty} \leqslant e^{-S}E_{0,\infty}$ and also
$E_{T,\infty} \leqslant e^{-T+h} E_{h,\infty}$
As $f''(t)=\rho(t),$ using Corollary \ref{corh0} along with Peano's Theorem implies that
$E_{h,3h}\leqslant e^{-3h/2}E_{0,h}$ whenever $h\leqslant h_0.$ Thence,
\begin{align*}
E_{h,\infty} = E_{h,3h}+E_{3h,\infty} \leqslant e^{-3h/2} E_{0,h}+e^{-2h}E_{h,\infty} \leqslant e^{-3h/2}E_{0,\infty},
\end{align*}
whence
\begin{align*}
E_{0,T}=E_{0,\infty}-E_{T,\infty} \geqslant \left(1-e^{-T-h/2}\right)E_{0,\infty} \geqslant e^{-S}E_{0,\infty} \geqslant E_{S,\infty}
\end{align*}
due to \eqref{est}. This is the desired inequality \eqref{ineq}. To finish the proof, we observe that for all $q\in(0, \frac{1}{2}),$ $x\in[0,1]$ and $i\in{\mathbb N},$ it is true that
$$\frac{x}{1-q^i} \leqslant \frac{1}{1-q^ix}$$
and thence, for $q\in(0, \frac{1}{2})$ and $x\in[0,1]$
$$
\prod_{j=0}^{k-1}\prod_{\ell=1}^{m-1} \frac{x}{1-q^{\ell+mj}}
\leqslant
\prod_{j=0}^{k-1}\prod_{\ell=1}^{m-1} \frac{1}{1-q^{\ell+mj}x}
\leqslant
\prod_{j=0}^{\infty}\prod_{\ell=1}^{m-1} \frac{1}{1-q^{\ell+mj}x}.
$$
The proof is complete.
\end{proof}

\begin{proof}[Proof of Theorem \ref{thmqr2}]
For given $\varepsilon>0,$ one can find $\delta>0$ such that $p_0(q;x)<\varepsilon/4$ for $x\in[1-2\delta, 1-\delta].$
Because of \eqref{sumpk}, for $N\in{\mathbb N},$ one may write:
\begin{align*}
\sum_{k=1}^N p_k(q;x) = 1 - p_0(q;x)-\sum_{k=N+1}^{\infty} p_k(q;x), \quad x\in[0,1).
\end{align*}
Notice that the series in \eqref{sumpk} converges uniformly on any closed subinterval of $[0,1),$
in particular, on $[0,1-\delta].$ Therefore, there exists $N_0\in{\mathbb N},$ such that
\begin{align*}
\sum_{k=N+1}^\infty p_k(q;x) < \frac{\varepsilon}{4}, \quad \text{for all} \quad x\in[0,1-\delta], \quad N>N_0.
\end{align*}
Hence, on $[1-2\delta, 1-\delta],$ there holds:
\begin{align}
\sum_{k=1}^N p_k(q;x) > 1-\frac{\varepsilon}{2}. \label{1Npk}
\end{align}
Apply this procedure to find $N_1, N_2$ and $\delta$ satisfying both
\begin{align*}
\sum_{k=1}^{N_1} p_k(q;x) > 1-\frac{\varepsilon}{2} \quad \text{and} \quad
\sum_{k=1}^{N_2} p_k(r;x) > 1-\frac{\varepsilon}{2}
\end{align*}
for $x\in[1-2\delta, 1-\delta].$ Setting $N=\max\{N_1, N_2\},$ one obtains
\begin{align*}
\sum_{k=1}^{N} [p_k(q;x)+p_k(r;x)] > 2-\varepsilon \quad \text{for} \quad x\in[1-2\delta, 1-\delta].
\end{align*}
At this point,   % set ${\mathbb I}_N(q):=\{1-q^k: k=1,2,\ldots, N\}$ and ${\mathbb I}_R(q)={\mathbb J}_q\setminus {\mathbb I}_N(q),$ and
for every $N\in \mathbb{N},$ consider a function $f_N\in C[0,1]$ such that $\|f_N\|=1$ and
$$\begin{array}{ll}
f_N(1-q^k) = 1 & \mathrm{when} \;\; \:\: k = 1,2,\dots ,  N,\\
f_N(1-r^k) = -1 & \mathrm{when} \;\; \:\: k =1,2,\dots ,  N,\\
%f_N(1-q^k) = 1 & \mathrm{for}\;\;m\nmid k,\:\:  k \leqslant N,\\
f_N(1-q^k) = f_N(1-r^k)= 0 & \mathrm{when}\;\; k \neq 1,2,\dots ,  N,.
\end{array}
$$
Then,
$$
\left(B_q-B_r\right)f(x)=\sum_{k=1}^{N} [p_k(q;x)+p_k(r;x)] > 2-\varepsilon \quad \text{for} \quad x\in[1-2\delta, 1-\delta].
$$
Since,
$$
\|B_q-B_r\| \geqslant \|\left(B_q-B_r\right)f(x)\| \geqslant 2-\varepsilon
$$
and as $\varepsilon >0$ has been chosen arbitrarily, one concludes that $\|B_q-B_r\| \geqslant 2.$
On the other hand, by the triangle inequality, one has $\|B_q-B_r\| \leqslant 2,$ and the statement follows.
\end{proof}

\begin{proof}[Proof of Theorem \ref{thmBqrmj}]

Let $f\in C[0,1].$ Then, for $x\in [0,1),$
\begin{align*}
\left(B_q-B_r\right)f(x)
&=\sum_{k=0}^\infty f(1-q^k) p_k(q;x)-\sum_{k=0}^\infty f(1-r^k) p_k(r;x) \\
&=\sum_{m|k} f(1-q^k) p_k(q;x)+\sum_{m\nmid k} f(1-q^k) p_k(q;x) \\
& \qquad -\sum_{j|k} f(1-r^k) p_k(r;x)-\sum_{j\nmid k} f(1-r^k) p_k(r;x) \\
&=\sum_{k=0}^\infty f(1-q^{mk}) \left[p_{mk}(q;x)-p_{jk}(r;x)\right] \\
&\qquad +\sum_{m\nmid k} f(1-q^k) p_k(q;x)-\sum_{j\nmid k} f(1-r^k) p_k(r;x)
\end{align*}
For each $N\in \mathbb{N}$, choose $f_N\in C[0,1]$ with $\|f_N\|=1$ in such a way that
$$\begin{array}{ll}
f_N(1-r^k) = -1 & \mathrm{for} \;\;j\nmid k, \:\: k \leqslant N,\\
f_N(1-q^k) = -1 & \mathrm{for} \;\;m | k, \:\: k \leqslant N,\\
f_N(1-q^k) = 1 & \mathrm{for}\;\;m\nmid k,\:\:  k \leqslant N,\\
f_N(1-q^k) = f_N(1-r^k)= 0 & \mathrm{for}\;\; k > N.
\end{array}
$$
Then,
\begin{align*}
\left(B_q-B_r\right)f_N(x)
&=\sum_{k=0}^N p_k(q;x)+\sum_{k=0}^N p_k(r;x)- 2\sum_{k=0}^{\lfloor N/m \rfloor} p_{mk}(q;x).
\end{align*}
Following the line of reasoning in the preceding proof and bearing in mind \eqref{limpmk}, one may opt for $\delta >0$ and $N\in \mathbb{N}$ such that, for every $\varepsilon >0,$
$$\sum_{k=0}^N p_k(q;x)+\sum_{k=0}^N p_k(r;x)- 2\sum_{k=0}^{\lfloor N/m \rfloor} p_{mk}(q;x) \geqslant 2-\frac{2}{m}-\varepsilon\quad \text{when} \quad x\in[1-2\delta, 1-\delta].$$ The statement is now  immediate.
\end{proof}

\begin{proof}[Proof of Theorem \ref{thmBqr1}]
Let $f\in C[0,1].$ Then
\begin{align*}
\left(B_q-B_r\right)f(x)
&=\sum_{k=0}^\infty f(1-q^k) p_k(q;x)-\sum_{k=0}^\infty f(1-r^k) p_k(r;x) \\
&=\sum_{m|k} f(1-q^k) p_k(q;x)+\sum_{m\nmid k} f(1-q^k) p_k(q;x)-\sum_{k=0}^\infty f(1-q^{mk}) p_k(q^m;x) \\
&=\sum_{k=0}^\infty f(1-q^{mk}) \left[p_{mk}(q;x)-p_{k}(q^m;x)\right]+\sum_{m\nmid k} f(1-q^k) p_k(q^m;x)
\end{align*}
Using Theorem \ref{thmfed} and \eqref{sumpk}, the last inequality becomes
\begin{align*}
\left|\left(B_q-B_r\right)f(x)\right|
&\leqslant 2\|f\|\left(1-\sum_{k=0}^\infty p_{mk}(q;x)\right), \quad \text{for} \quad x\in[0, 1).
\end{align*}
Now, by \eqref{limpmk}, for all $\varepsilon >0$ there exists $\delta>0$ such that
$$
\sum_{k=0}^\infty p_{mk}(q;x) > \frac{1}{m}-\frac{\varepsilon}{2}, \quad \text{for} \quad x\in[1-2\delta, 1-\delta].
$$
Therefore,
\begin{align*}
\left|\left(B_q-B_r\right)f(x)\right|
&\leqslant 2\|f\|\left(1-\frac{1}{m}+\frac{\varepsilon}{2}\right), \quad \text{for} \quad x\in[1-2\delta, 1-\delta].
\end{align*}
This implies that $\|B_{q}-B_{r}\| \leqslant 2-2/m.$ Together with Theorem \ref{thmBqrmj}, this yields the statement.
\end{proof}

\section*{Acknowledgments}
During the work on this paper, the authors were lucky to see the discussion on Mathoverflow \cite{fedja} concerning  inequality \eqref{eqfed} in the case $m=2.$ We would like to express our sincere gratitude to all MO users,   who participated in this fruitful discussion, especially to the user whose nickname is `fedja' and whose grasp of the subject was of a significant inspiration.
We are also pleased to thank Prof. Alexandre Eremenko (Purdue University, USA) for his encouragement and valuable help throughout the entire process of our work.

\end{document}